\documentclass[reqno,a4paper]{amsart}
\usepackage{amssymb,url}
\usepackage[hypertex]{hyperref}
\allowdisplaybreaks[4]
\DeclareMathOperator{\td}{d\mspace{-2mu}}
\theoremstyle{plain}
\newtheorem{lem}{Lemma}
\newtheorem{thm}{Theorem}
\newtheorem{cor}{Corollary}
\theoremstyle{remark}
\newtheorem{rem}{Remark}
\numberwithin{equation}{section}

\begin{document}

\title[Complete monotonicity of functions involving polygamma functions]
{Complete monotonicity of some functions involving polygamma functions}

\author[F. Qi]{Feng Qi}
\address[F. Qi]{Research Institute of Mathematical Inequality Theory, Henan Polytechnic University, Jiaozuo City, Henan Province, 454010, China}
\email{\href{mailto: F. Qi <qifeng618@gmail.com>}{qifeng618@gmail.com}, \href{mailto: F. Qi <qifeng618@hotmail.com>}{qifeng618@hotmail.com}, \href{mailto: F. Qi <qifeng618@qq.com>}{qifeng618@qq.com}}
\urladdr{\url{http://qifeng618.spaces.live.com}}

\author[S. Guo]{Senlin Guo}
\address[S. Guo]{Department of Mathematics, Zhongyuan University of Technology, Zhengzhou City, Henan Province, 450007, China}
\email{\href{mailto: S. Guo <sguo@hotmail.com>}{sguo@hotmail.com}, \href{mailto: S. Guo <senlinguo@gmail.com>}{senlinguo@gmail.com}}

\author[B.-N. Guo]{Bai-Ni Guo}
\address[B.-N. Guo]{School of Mathematics and Informatics,
Henan Polytechnic University, Jiaozuo City, Henan Province, 454010, China}
\email{\href{mailto: B.-N. Guo <bai.ni.guo@gmail.com>}{bai.ni.guo@gmail.com}, \href{mailto: B.-N. Guo <bai.ni.guo@hotmail.com>}{bai.ni.guo@hotmail.com}}
\urladdr{\url{http://guobaini.spaces.live.com}}

\begin{abstract}
In the present paper, we establish necessary and sufficient conditions for the functions $x^\alpha\bigl\lvert\psi^{(i)}(x+\beta)\bigr\lvert$ and $\alpha\bigl\lvert\psi^{(i)}(x+\beta)\bigr\lvert-x\bigl\lvert\psi^{(i+1)}(x+\beta)\bigr\lvert$ respectively to be monotonic and completely monotonic on $(0,\infty)$, where $i\in\mathbb{N}$, $\alpha>0$ and $\beta\ge0$ are scalars, and $\psi^{(i)}(x)$ are polygamma functions.
\end{abstract}



\thanks{The first author was partially supported by the China Scholarship Council}

\thanks{This paper was typeset using \AmS-\LaTeX}

\maketitle

\section{Introduction}

\subsection{}
Recall~\cite[Chapter~XIII]{mpf-93} and~\cite[Chapter~IV]{widder} that a function $f(x)$ is said to be completely monotonic on an interval $I\subseteq\mathbb{R}$ if $f(x)$ has derivatives of all orders on $I$ and
\begin{equation}
0\le(-1)^{k}f^{(k)}(x)<\infty
\end{equation}
holds for all $k\geq0$ on $I$. This definition was introduced in 1921 by F. Hausdorff in~\cite{Hausdorff-1921}, who called such functions ``total monoton''.
\par
The celebrated Bernstein-Widder Theorem~\cite[p.~161]{widder} states that a function $f(x)$ is completely monotonic on $(0,\infty)$ if and only if
\begin{equation}\label{converge}
f(x)=\int_0^\infty e^{-xs}\td\mu(s),
\end{equation}
where $\mu$ is a nonnegative measure on $[0,\infty)$ such that the integral~\eqref{converge} converges for all $x>0$. This means that a function $f(x)$ is completely monotonic on $(0,\infty)$ if and only if it is a Laplace transform of the measure $\mu$.
\par
The most important properties of completely monotonic functions can be found in~\cite[Chapter~XIII]{mpf-93}, \cite[Chapter~IV]{widder}, \cite{Berg-book-84, haer} and the related references therein.
\par
The completely monotonic functions have applications in different branches of mathematical sciences. For example, they play some role in combinatorics~\cite{Ball-94}, numerical and asymptotic analysis~\cite{Frenzen-art-87, Wimp-book-81}, physics~\cite{Day-art-70, Feller-book-prob}, potential theory~\cite{Berg-Forest-book-75}, and probability theory~\cite{Bondesson-book-92, Feller-book-prob, Kimberling-art-74}.

\subsection{}
It is well-known~\cite{abram, wange, wang} that the classical Euler gamma function may be defined for $x>0$ by
\begin{equation}\label{egamma}
\Gamma(x)=\int^\infty_0t^{x-1} e^{-t}\td t.
\end{equation}
The logarithmic derivative of $\Gamma(x)$, denoted by $\psi(x)=\frac{\Gamma'(x)}{\Gamma(x)}$, is called psi function, and $\psi^{(k)}(x)$ for $k\in \mathbb{N}$ are called polygamma functions.
\par
It should be common knowledge~\cite{abram, wange, wang} that the special functions $\Gamma(x)$, $\psi(x)$ and $\psi^{(i)}(x)$ for $i\in\mathbb{N}$ are important and basic and that they have much extensive applications in mathematical sciences.

\subsection{}\label{ACAP-323-sec-1.3}
In~\cite[Lemma~1]{alzeraeq}, it was shown that the functions $x^c\bigl\lvert\psi^{(k)}(x)\bigr\lvert$ for $k\in\mathbb{N}$ and $c\in\mathbb{R}$ are strictly decreasing (or strictly increasing, respectively) on $(0,\infty)$ if and only if $c\le k$ (or $c\ge k+1$, respectively).
\par
In~\cite[Theorem~4.14]{forum-alzer}, it was obtained that the function $x^c\bigl\lvert\psi^{(k)}(x)\bigr\lvert$ for $k\in\mathbb{N}$ and $c\in\mathbb{R}$ is strictly convex on $(0,\infty)$ if and only if either $c\le k$, or $c=k+1$, or $c\ge k+2$. In~\cite[Remark~4.15]{forum-alzer}, it was pointed out that there does not exist a real number $c$ such that the function $x^c\bigl\lvert\psi^{(k)}(x)\bigr\lvert$ for $k\in\mathbb{N}$ is concave on $(0,\infty)$.
\par
In~\cite[Lemma~2.2]{forum-alzer} and~\cite[Lemma~5]{alsina-tomas-sandor.tex-rgmia}, the functions $x^\alpha\bigl\lvert\psi^{(i)}(x+1)\bigr\lvert$ for $i\in\mathbb{N}$ are proved to be strictly increasing (or strictly decreasing, respectively) on $(0,\infty)$ if and only if $\alpha\ge i$ (or $\alpha\le0$, respectively).
\par
In~\cite[Lemma~2.1]{gao-sub-add-psi}, the function $x\psi'(x+a)$ is proved to be strictly increasing on $[0,\infty)$ for $a\ge1$.
\par
Motivated by the above results, the first and third authors considered in~\cite{subadditive-qi.tex} the monotonicity of a more general function $x^\alpha\bigl\lvert\psi^{(i)}(x+\beta)\bigr\lvert$ and the complete monotonicity of several related functions as follows: For $i\in\mathbb{N}$, $\alpha>0$ and $\beta\ge0$,
\begin{enumerate}
\item
the function $x^\alpha\bigl\lvert\psi^{(i)}(x+\beta)\bigr\lvert$ is strictly increasing on $(0,\infty)$ if $(\alpha,\beta)\in\bigl\{\alpha\ge i,\frac12\le\beta<1\bigr\}\cup\bigl\{\alpha\ge i,\beta\ge\frac{\alpha-i+1}2\bigr\} \cup\bigl\{\alpha\ge i+1,\beta\le\frac{\alpha-i+1}2\bigr\}$ and only if $\alpha\ge i$;
\item
the function $\frac{\alpha}{x}\bigl\lvert\psi^{(i)}(x)\bigr\lvert-\bigl\lvert\psi^{(i+1)}(x)\bigl\lvert$ is completely monotonic on $(0,\infty)$ if and only if $\alpha\ge i+1$;
\item
the function $\bigl\lvert\psi^{(i+1)}(x)\bigl\lvert-\frac{\alpha}{x}\bigl\lvert\psi^{(i)}(x)\bigl\lvert$ is completely monotonic on $(0,\infty)$ if and only if $0<\alpha\le i$;
\item
the function $\frac{\alpha}{x}\bigl\lvert\psi^{(i)}(x+1)\bigl\lvert-\bigl\lvert\psi^{(i+1)}(x+1)\bigl\lvert$ is completely monotonic on $(0,\infty)$ if and only if $\alpha\ge i$;
\item
the function $\frac{\alpha}{x}\bigl\lvert\psi^{(i)}(x+\beta)\bigl\lvert -\bigl\lvert\psi^{(i+1)}(x+\beta)\bigl\lvert$ on $(0,\infty)$ is completely monotonic if $(\alpha,\beta)\in\bigl\{\alpha\ge i+1,\beta\le\frac{\alpha-i+1}2\bigr\}
\cup\bigl\{i\le \alpha\le i+1,\beta\ge\frac{\alpha-i+1}2\bigr\}\cup\bigl\{i\le\alpha\le \frac{(i+1) (i+4\beta-2)} {i+2\beta},\frac12\le\beta<1\bigr\}$ and only if $\alpha\ge i$;
\item
the function $\alpha\bigl\lvert\psi^{(i)}(x+\beta)\bigl\lvert-x\bigl\lvert\psi^{(i+1)}(x+\beta)\bigl\lvert$ is completely monotonic on $(0,\infty)$ if $(\alpha,\beta)\in\bigl\{i\le \alpha\le i+1,\beta\ge\frac{\alpha-i+1}2\bigr\} \cup\bigl\{\alpha\ge i+1,\beta\le\frac{\alpha-i+1}2\bigr\}$ and only if $\alpha\ge i$.
\end{enumerate}

\subsection{}\label{sec-1.4}
The first aim of this paper is to present necessary and sufficient conditions for the function $x^\alpha\bigl\lvert\psi^{(i)}(x+\beta)\bigl\lvert$ to be monotonic on $(0,\infty)$, which can be summarized as the following Theorem~\ref{itsf-gao-gen-3}.

\begin{thm}\label{itsf-gao-gen-3}
Let $i\in\mathbb{N}$, $\alpha\in\mathbb{R}$ and $\beta\ge0$.
\begin{enumerate}
\item
The function $x^\alpha\bigl\lvert\psi^{(i)}(x)\bigl\lvert$ is strictly increasing \textup{(}or strictly decreasing, respectively\textup{)} on $(0,\infty)$ if and only if $\alpha\ge i+1$ \textup{(}or $\alpha\le i$, respectively\textup{)}.
\item
For $\beta\ge\frac12$, the function $x^\alpha\bigl\lvert\psi^{(i)}(x+\beta)\bigl\lvert$ is
strictly increasing on $[0,\infty)$ if and only if $\alpha\ge i$.
\item
Let $\delta:(0,\infty)\to\bigl(0,\frac12\bigr)$ be defined by
\begin{equation}\label{delta-dfn}
\delta(t)=\frac{e^t(t-1)+1}{(e^t-1)^2}
\end{equation}
and $\delta^{-1}:\bigl(0,\frac12\bigr)\to(0,\infty)$
stand\texttt{} for the inverse function of $\delta$. If $0<\beta<\frac12$ and
\begin{equation}\label{alpha-beta-inverse}
\alpha\ge i+1 -\biggl[\frac{e^{\delta^{-1}(\beta)}} {e^{\delta^{-1}(\beta)}-1}
+\beta-1\biggr]\delta^{-1}(\beta),
\end{equation}
then the function $x^\alpha\bigl\lvert\psi^{(i)}(x+\beta)\bigl\lvert$ is strictly increasing
on $(0,\infty)$.
\end{enumerate}
\end{thm}

As a by-product of the proof of Theorem~\ref{itsf-gao-gen-3}, lower and upper bounds for infinite series whose coefficients involve Bernoulli numbers may be derived as follows.

\begin{cor}\label{bernou-thm}
Let $0<\beta<\frac12$ and $\delta^{-1}$ be the inverse function of $\delta$
defined by~\eqref{delta-dfn}. Then the following inequalities holds:
\begin{gather}
\frac12>\sum_{k=1}^\infty B_{2k}\frac{t^{2k-1}}{(2k-1)!}>0,\\
\frac{t}2>\sum_{k=0}^\infty B_{2k+2}\frac{t^{2k+2}}{(2k+2)!}
>\max\biggl\{0,\frac{t}2-1\biggr\},\\
\sum_{k=0}^\infty B_{2k+2}\frac{t^{2k+2}}{(2k+2)!}
>\biggl(\frac12-\beta\biggr)t+
\biggl[\frac{e^{\delta^{-1}(\beta)}} {e^{\delta^{-1}(\beta)}-1}
-\beta+1\biggr]\delta^{-1}(\beta)-1,
\end{gather}
where $t\in(0,\infty)$ and $B_n$ for $n\ge0$ represent Bernoulli numbers which may be defined~\cite{abram, wange, wang} by
\begin{equation}
\begin{aligned}
\frac{x}{e^x-1}&=\sum_{n=0}^\infty \frac{B_n}{n!}x^n =1-\frac{x}2+\sum_{j=1}^\infty B_{2j}\frac{x^{2j}}{(2j)!}, &\vert x\vert &<2\pi.
\end{aligned}
\end{equation}
\end{cor}

\subsection{}\label{sec-1.5}
The second aim of this paper is to establish necessary and sufficient conditions for the function $\alpha\bigl\lvert\psi^{(i)}(x+\beta)\bigl\lvert-x\bigl\lvert\psi^{(i+1)}(x+\beta)\bigl\lvert$ to be completely monotonic on $(0,\infty)$, which may be stated as the following Theorem~\ref{psi-plus-comp-mon-3}.

\begin{thm}\label{psi-plus-comp-mon-3}
Let $i\in\mathbb{N}$, $\alpha\in\mathbb{R}$ and $\beta\ge0$.
\begin{enumerate}
\item
The function
\begin{equation}\label{alpha-psi-x}
\alpha\bigl\lvert\psi^{(i)}(x)\bigl\lvert-x\bigl\lvert\psi^{(i+1)}(x)\bigl\lvert
\end{equation}
is completely monotonic on $(0,\infty)$ if and only if $\alpha\ge i+1$.
\item
The negative of the function~\eqref{alpha-psi-x} is completely monotonic on $(0,\infty)$ if and only if $\alpha\le i$.
\item
If $\beta\ge\frac12$, then the function
\begin{equation}\label{alpha-psi-x-beta}
\alpha\bigl\lvert\psi^{(i)}(x+\beta)\bigl\lvert -x\bigl\lvert\psi^{(i+1)}(x+\beta)\bigl\lvert
\end{equation}
is completely monotonic on $(0,\infty)$ if and only if
$\alpha\ge i$.
\item
If $0<\beta<\frac12$ and the inequality~\eqref{alpha-beta-inverse} is valid, then the function~\eqref{alpha-psi-x-beta} is completely monotonic on $(0,\infty)$.
\end{enumerate}
\end{thm}

As immediate consequences of Theorem~\ref{psi-plus-comp-mon-3}, the following corollary is obtained.

\begin{cor}\label{cor-sub-3}
Let $i\in\mathbb{N}$, $\alpha\in\mathbb{R}$ and $\beta\ge0$.
\begin{enumerate}
\item
The function
\begin{equation}\label{alpha-psi/x}
\frac{\alpha}x\bigl\lvert\psi^{(i)}(x)\bigl\lvert-\bigl\lvert\psi^{(i+1)}(x)\bigl\lvert
\end{equation}
is completely monotonic on $(0,\infty)$ if and only if $\alpha\ge i+1$.
\item
The negative of the function~\eqref{alpha-psi/x} is completely monotonic on $(0,\infty)$ if and only if $\alpha\le i$.
\item
If $\beta\ge\frac12$, then the function
\begin{equation}\label{alpha-psi/x-beta}
\frac{\alpha}x\bigl\lvert\psi^{(i)}(x+\beta)\bigl\lvert-\bigl\lvert\psi^{(i+1)}(x+\beta)\bigl\lvert
\end{equation}
is completely monotonic on $(0,\infty)$ if and only if $\alpha\ge i$.
\item
If $0<\beta<\frac12$ and the inequality~\eqref{alpha-beta-inverse} holds true, then the function~\eqref{alpha-psi/x-beta} is completely monotonic on $(0,\infty)$.
\end{enumerate}
\end{cor}

\section{Remarks}

Before proving the above theorems and corollaries, we give several remarks about Theorem~\ref{itsf-gao-gen-3}, Theorem~\ref{psi-plus-comp-mon-3} and their applications.

\begin{rem}
Since $\lim_{\beta\to0^+}[\beta\delta^{-1}(\beta)]=0$ and that the inverse function $\delta^{-1}$ is decreasing from $\bigl(0,\frac12\bigr)$ onto $(0,\infty)$, we claim that
\begin{equation}\label{0-1-lim}
0<\biggl[\frac{e^{\delta^{-1}(\beta)}} {e^{\delta^{-1}(\beta)}-1}
+\beta-1\biggr]\delta^{-1}(\beta)<1,\quad \beta\in(0,1).
\end{equation}
Indeed, if replacing $\delta^{-1}(\beta)$ by $s$, the middle term in~\eqref{0-1-lim} becomes $\frac{s^2e^s}{(e^s-1)^2}$ which is decreasing from $(0,\infty)$ onto $(0,1)$. This implies that the condition~\eqref{alpha-beta-inverse} is not only sufficient but also necessary in Theorem~\ref{itsf-gao-gen-3} and Theorem~\ref{psi-plus-comp-mon-3}.
\end{rem}

\begin{rem}\label{acap-323-rem-1}
As mentioned in Section~\ref{ACAP-323-sec-1.3}, some conclusions in Theorem~\ref{itsf-gao-gen-3} have been applied in nearby fields.
\begin{enumerate}
\item
The first conclusion in Theorem~\ref{itsf-gao-gen-3} was utilized in~\cite[Theorem~2]{alzeraeq} to obtain a functional inequality concerning polygamma functions: For $k\ge1$ and $n\ge2$, the inequality
\begin{equation}
\bigl\vert\psi^{(k)}\bigl(M_n^{[r]}(x_\nu;p_\nu)\bigr)\bigr\vert\le M_n^{[s]}\bigl(\bigl\vert\psi^{(k)}(x_\nu)\bigr\vert;p_\nu\bigr)
\end{equation}
holds if and only if either $r\ge0$ and $s\ge-\frac{r}{k+1}$ or $r<0$ and $s\ge-\frac{r}k$, where $x_\nu>0$ and $p_\nu>0$ with $\sum_{\nu=1}^np_\nu=1$, and
\begin{equation}
M_n^{[t]}(x_\nu;p_\nu)=
\begin{cases}
\biggl(\sum\limits_{\nu=1}^np_\nu x_\nu^t\biggr)^{1/t},&t\ne0\\[0.8em]
\prod\limits_{\nu=1}^nx_\nu^{p_\nu},& t=0
\end{cases}
\end{equation}
stands for the discrete weighted power means.
\item
Some applications of the first two conclusions in Theorem~\ref{itsf-gao-gen-3} were carried out in~\cite{forum-alzer} as follows.
\begin{enumerate}
\item
The first conclusion in Theorem~\ref{itsf-gao-gen-3} was applied in~\cite[Theorem~4.9]{forum-alzer} to obtain that the inequalities
\begin{equation}
\biggl(1+\frac\alpha{x+s}\biggr)^n<\frac{\psi^{(n)}(x+s)}{\psi^{(n)}(x+1)} <\biggl(1+\frac\beta{x+s}\biggr)^n
\end{equation}
hold for $n\in\mathbb{N}$ and $s\in(0,1)$ with the best possible constants
\begin{equation}
\alpha=1-s\quad\text{and}\quad \beta=s\biggl[\frac{\psi^{(n)}(s)}{\psi^{(n)}(1)}\biggr]^{1/n}-s.
\end{equation}
\item
The special cases for $\beta=1$ of the third conclusion in Theorem~\ref{itsf-gao-gen-3} was employed in~\cite[Theorem~4.8]{forum-alzer} to derive that the inequalities
\begin{equation}
(n-1)!\exp\biggl[\frac\alpha{x}-n\psi(x)\biggr]<\bigl\vert\psi^{(n)}(x)\bigr\vert <(n-1)!\exp\biggl[\frac\beta{x}-n\psi(x)\biggr]
\end{equation}
hold for $n\in\mathbb{N}$ and $x>0$ if and only if $\alpha\le-n$ and $\beta\ge0$.
\item
Moreover, the convexity of $x^c\bigl\lvert\psi^{(k)}(x)\bigl\lvert$ was used in~\cite[Theorem~4.16]{forum-alzer} to establish that the double inequalities
\begin{equation}
\alpha\biggl(\frac1x-\frac1y\biggr)< x^n\bigl\vert\psi^{(n)}(x)\bigr\vert- y^n\bigl\vert\psi^{(n)}(y)\bigr\vert<\beta\biggl(\frac1x-\frac1y\biggr)
\end{equation}
hold for $n\in\mathbb{N}$ and $y>x>0$ with the best possible constants $\alpha=\frac{n!}2$ and $\beta=n!$.
\end{enumerate}
\item
The special cases for $\beta=1$ of the third conclusion in Theorem~\ref{itsf-gao-gen-3}, the monotonic properties of the functions
$$
\begin{aligned}
& x^{2i}\psi^{(2i)}(1+x), & & & & x^{2i+1}\psi^{(2i)}(1+x),\\
& x^{2i-1}\psi^{(2i-1)}(1+x) & & \text{and} & & x^{2i}\psi^{(2i-1)}(1+x)
\end{aligned}
$$
on $[0,\infty)$, were used in~\cite{alsina-tomas-sandor.tex-rgmia, alsina-tomas-sandor.tex} to establish the monotonic, logarithmically convex, completely monotonic properties of the functions
\begin{equation}
\frac{[\Gamma(1+x)]^y}{\Gamma(1+xy)}\quad \text{and}\quad \frac{\Gamma(1+y)[\Gamma(1+x)]^y}{\Gamma(1+xy)}
\end{equation}
or their first and second logarithmic derivatives.
\item
The special case for $\beta\ge1$ and $i=1$ of the third conclusion in Theorem~\ref{itsf-gao-gen-3} was employed in~\cite[Theorem~3.1]{gao-sub-add-psi} to reveal the subadditive property of the function $\psi(a+e^x)$ on $(-\infty,\infty)$.
\end{enumerate}
\end{rem}

\begin{rem}\label{acap-323-rem-2}
The former two conclusions in Theorem~\ref{itsf-gao-gen-3}, which were ever circulated in the preprint~\cite{subadditive-qi-3.tex}, have been employed in the proofs of~\cite[Theorem~1.4]{ATSC-ITSF.tex} and~\cite[Theorem~2.2 and Theorem~2.3]{note-on-neuman-ITSF-simplified.tex}.
\begin{enumerate}
\item
The special cases for $\alpha=i$ and $\beta=1$ of the second conclusion in Theorem~\ref{itsf-gao-gen-3} were made use of to procure the following theorem.

\begin{thm}[{\cite[Theorem~1.4]{ATSC-ITSF.tex}}]\label{continue-thm}
The function
\begin{equation}\label{G(s,t;x)}
G_{s,t}(x)=\frac{[\Gamma(1+tx)]^s}{[\Gamma(1+sx)]^t}
\end{equation}
for $x,s,t\in\mathbb{R}$ such that $1+sx>0$ and $1+tx>0$ with $s\ne t$ has the following properties:
\begin{enumerate}
\item
For $t>s>0$ and $x\in(0,\infty)$, $G_{s,t}(x)$ is an increasing function and a logarithmically completely monotonic function of second order in $x$;
\item
For $t>s>0$ and $x\in\bigl(-\frac1t,0\bigr)$, $G_{s,t}(x)$ is a logarithmically completely monotonic function in $x$;
\item
For $s<t<0$ and $x\in(-\infty,0)$, $G_{s,t}(x)$ is a decreasing function and a logarithmically absolutely monotonic function of second order in $x$;
\item
For $s<t<0$ and $x\in\bigl(0,-\frac1s\bigr)$, $G_{s,t}(x)$ is a logarithmically completely monotonic function in $x$;
\item
For $s<0<t$ and $x\in\bigl(-\frac1t,0\bigr)$, $G_{t,s}(x)$ is an increasing function and a logarithmically absolutely convex function in $x$;
\item
For $s<0<t$ and $x\in\bigl(0,-\frac1s\bigr)$, $G_{t,s}(x)$ is a decreasing function and a logarithmically absolutely convex function in $x$.
\end{enumerate}
\end{thm}

\item
The first two conclusions in Theorem~\ref{itsf-gao-gen-3} were hired in \cite{note-on-neuman-ITSF-simplified.tex}, a simplified version of the preprint~\cite{note-on-neuman.tex}, to derive the following two theorems.

\begin{thm}[{\cite[Theorem~2.2]{note-on-neuman-ITSF-simplified.tex}}]
For $b>a>0$ and $i\in\mathbb{N}$, the function $\frac{[\Gamma(bx)]^a}{[\Gamma(ax)]^b}$ is $(2i+1)$-log-convex and $(2i)$-log-concave with respect to $x\in(0,\infty)$.
\end{thm}

\begin{thm}[{\cite[Theorem~2.3]{note-on-neuman-ITSF-simplified.tex}}]
For $b>a>0$, $i\in\mathbb{N}$ and $\beta\ge\frac12$, the function $\frac{[\Gamma(bx+\beta)]^a}{[\Gamma(ax+\beta)]^b}$ is $(2i+1)$-log-concave and $(2i)$-log-convex with respect to $x\in(0,\infty)$.
\end{thm}
\end{enumerate}
\par
For exact definitions of the terminologies such as ``logarithmically completely monotonic function of $k$-th order'', ``logarithmically absolutely monotonic function of $k$-th order'', ``$k$-log-convex function'' and ``logarithmically absolutely convex function'', see~\cite[Definition~1.1, Definition~1.2 and Definition~1.3]{ATSC-ITSF.tex}, or the first paragraph of~\cite{note-on-neuman-ITSF-simplified.tex}, or related texts in~\cite{schur-complete}.
\end{rem}

\begin{rem}\label{acap-323-rem-3}
The former two conclusions in Theorem~\ref{psi-plus-comp-mon-3}, which were also issued in the preprint~\cite{subadditive-qi-3.tex}, have also been applied in the proofs of~\cite[Theorem~1]{new-upper-kershaw-JCAM.tex} and~\cite[Theorem~1]{new-upper-kershaw-2.tex-mia}.
\begin{enumerate}
\item
The first result in Theorem~\ref{psi-plus-comp-mon-3} was utilized in \cite[Theorem~1]{new-upper-kershaw-JCAM.tex} to procure upper bounds for the ratio of two gamma functions and the divided differences of the psi and polygamma functions as follows.

\begin{thm}[{\cite[Theorem~1]{new-upper-kershaw-JCAM.tex}}]\label{upper-ratio-1}
For $a>0$ and $b>0$ with $a\ne b$, inequalities
\begin{equation}\label{identric-kershaw-ineq-equiv}
\biggl[\frac{\Gamma(a)}{\Gamma(b)}\biggr]^{1/(a-b)}\le e^{\psi(I(a,b))}
\end{equation}
and
\begin{equation}\label{batir-psi-ineq-ref-equiv}
\frac{(-1)^{n}\bigl[\psi^{(n-1)}(a) -\psi^{(n-1)}(b)\bigr]}{a-b} \le(-1)^n\psi^{(n)}(I(a,b))
\end{equation}
hold true, where $n\in\mathbb{N}$ and
\begin{equation}
I(a,b)=\frac1e\biggl(\frac{b^b}{a^a}\biggr)^{1/(b-a)}
\end{equation}
represents the identric or exponential mean.
\end{thm}
\item
The first two results in Theorem~\ref{psi-plus-comp-mon-3} were employed in \cite[Theorem~1]{new-upper-kershaw-2.tex-mia} to acquire lower bounds for the ratio of two gamma functions and the divided differences of the psi and polygamma functions and to refine the inequality~\eqref{batir-psi-ineq-ref-equiv}.

\begin{thm}[{\cite[Theorem~1]{new-upper-kershaw-2.tex-mia}}]\label{upper-ratio-2}
For $a>0$ and $b>0$ with $a\ne b$, the inequality
\begin{equation}\label{new-upper-main-ineq}
(-1)^i\psi^{(i)}(L_\alpha(a,b))\le \frac{(-1)^i}{b-a}\int_a^b\psi^{(i)}(u)\td u \le(-1)^i\psi^{(i)}(L_\beta(a,b))
\end{equation}
holds if $\alpha\le-i-1$ and $\beta\ge-i$, where $i$ is a nonnegative integer and
\begin{equation}\label{L_p(a,b)}
L_p(a,b)=
\begin{cases}
\left[\dfrac{b^{p+1}-a^{p+1}}{(p+1)(b-a)}\right]^{1/p},&p\ne-1,0\\[1em]
\dfrac{b-a}{\ln b-\ln a},&p=-1\\
I(a,b),&p=0
\end{cases}
\end{equation}
stands for the generalized logarithmic mean of order $p\in\mathbb{R}$.
\end{thm}
\end{enumerate}
\par
The topic of bounding the ratio of two gamma functions has a history of at least sixty years since \cite{wendel}. For more information on its history, backgrounds, motivations and recent developments, please refer to, for example, \cite{alzeraeq, forum-alzer, Ismail-Muldoon-119, Muldoon-78, notes-best-simple-equiv.tex, notes-best-simple-open.tex, notes-best-simple.tex, sandor-gamma-3-note.tex-amc, new-upper-kershaw-JCAM.tex, new-upper-kershaw-2.tex-mia, Zhang-Morden, Zhang-Xu-Situ}, especially to the expository and survey preprint~\cite{bounds-two-gammas.tex} in which plentiful references are collected. For knowledge of mean values, please refer to the celebrated book~\cite{bullenmean} or the paper~\cite{royal-98}.
\end{rem}

\begin{rem}
Recall~\cite{alzer-koumandos, alzer-rus, gao-sub-add-psi} that a function $f(x)$ is said to be subadditive on $I$ if the inequality
\begin{equation}\label{sub-dfn-ineq}
f(x+y)\le f(x)+f(y)
\end{equation}
holds for all $x,y\in I$ with $x+y\in I$. If the inequality~\eqref{sub-dfn-ineq} is reversed, then $f(x)$ is called superadditive on $I$.
\par
The subadditive and superadditive functions play important roles in the theory of differential equations, in the study of semi-groups, in number theory, in the theory of convex bodies, and the like. See~\cite{alzer-koumandos, Alzer-Ruehr, alzer-rus} and the related references therein.
\par
Some subadditive or superadditive properties of the gamma, psi and polygamma functions have been discovered as follows.
\par
In \cite{Alzer-Ruehr}, the function $\psi(a+x)$ is proved to be sub-multiplicative with respect to $x\in[0,\infty)$ if and only if $a\ge a_0$, where $a_0$ denotes the only positive real number which satisfies $\psi(a_0)=1$.
\par
In~\cite{alzer-rus}, the function $[\Gamma(x)]^\alpha$ was proved to be subadditive on $(0,\infty)$ if and only if $\frac{\ln2}{\ln\Delta}\le\alpha\le0$, where $\Delta=\min_{x\ge0}\frac{\Gamma(2x)}{\Gamma(x)}$.
\par
In~\cite[Lemma~2.4]{forum-alzer}, the function $\psi(e^x)$ was proved to be strictly concave on $\mathbb{R}$.
\par
In~\cite[Theorem~3.1]{gao-sub-add-psi}, the function $\psi(a+e^x)$ is proved to be subadditive on $(-\infty,\infty)$ if and only if $a\ge c_0$, where $c_0$ is the only positive zero of $\psi(x)$.
\par
In~\cite[Theorem~1]{Extension-alzer-AMEN.tex}, among other things, it was presented that the function $\psi^{(k)}(e^x)$ for $k\in\mathbb{N}$ is concave $($or convex, respectively$)$ on $\mathbb{R}$ if $k=2n-2$ $($or $k=2n-1$, respectively$)$ for $n\in\mathbb{N}$.
\par
By the aid of the monotonicity of the function $x^\alpha\bigl\lvert\psi^{(i)}(x+\beta)\bigl\lvert$ in Theorem~\ref{itsf-gao-gen-3}, the following subadditive and superadditive properties of the function $\bigl\lvert\psi^{(i)}(e^x)\bigl\lvert$ for $i\in\mathbb{N}$ were acquired recently.

\begin{thm}[\cite{subadditive-qi-2.tex-rgmia}]\label{sub-sup-add}
For $i\in\mathbb{N}$, the function $\bigl\lvert\psi^{(i)}(e^x)\bigl\lvert$ is superadditive on $(-\infty,\ln\theta_0)$ and subadditive on $(\ln\theta_0,\infty)$, where $\theta_0\in(0,1)$ is the unique root of the equation $2\bigl\lvert\psi^{(i)}(\theta)\bigl\lvert=\bigl\lvert\psi^{(i)}(\theta^2)\bigl\lvert$.
\end{thm}
\end{rem}

\begin{rem}\label{acap-323-rem-4}
The second conclusion in Theorem~\ref{psi-plus-comp-mon-3} was cited in~\cite[Lemma~2.4]{Zhang-Morden} and~\cite[Remark~2.3]{Zhang-Xu-Situ}.
\end{rem}

\begin{rem}
Theorem~\ref{sub-sup-add} and the facts mentioned in Remark~\ref{acap-323-rem-2} to Remark~\ref{acap-323-rem-4} show the potential applicability of Theorem~\ref{itsf-gao-gen-3} and Theorem~\ref{psi-plus-comp-mon-3} convincingly.
\end{rem}

\begin{rem}
In passing, we recollect the notion ``logarithmically completely monotonic function'' which is equivalent to the logarithmically completely monotonic function of $0$-th order mentioned in Remark~\ref{acap-323-rem-2}. A function $f(x)$ is said to be logarithmically completely monotonic on an interval $I\subseteq\mathbb{R}$ if it has derivatives of all orders on $I$ and its logarithm $\ln f$ satisfies
\begin{equation}
0\le(-1)^k[\ln f(x)]^{(k)}<\infty
\end{equation}
for $k\in\mathbb{N}$ on $I$. By looking through the database \href{http://www.ams.org/mathscinet/}{MathSciNet}, we find that this phrase was first used in~\cite{Atanassov}, but with no a word to explicitly define it. Thereafter, it seems to have been ignored by the mathematical community. In early 2004, this terminology was recovered in~\cite{minus-one} and it was immediately referenced in \cite{auscm-rgmia}, the preprint of the paper~\cite{e-gam-rat-comp-mon}. A natural question that one may ask is: Whether is this notion trivial or not? In \cite[Theorem~4]{minus-one}, it was proved that all logarithmically completely monotonic functions are also completely monotonic, but not conversely. This result was formally published when revising~\cite{compmon2}. Hereafter, this conclusion and its proofs were dug in~\cite{CBerg, Gao-0709.1126v2-Arxiv, grin-ismail, schur-complete} once and again. Furthermore,  in the paper~\cite{CBerg}, the logarithmically completely monotonic functions on $(0,\infty)$ were characterized as the infinitely divisible completely monotonic functions studied in~\cite{horn} and all Stieltjes transforms were proved to be logarithmically completely monotonic on $(0,\infty)$, where a function $f(x)$ defined on $(0,\infty)$ is called a Stieltjes transform if it can be of the form
\begin{equation}
f(x)=a+\int_0^\infty\frac1{s+x}{\td\mu(s)}
\end{equation}
for some nonnegative number $a$ and some nonnegative measure $\mu$ on $[0,\infty)$ satisfying $\int_0^\infty\frac1{1+s}\td\mu(s)<\infty$. For more information, please refer to~\cite{CBerg}.
\par
It is remarked that many completely monotonic functions founded in a lot of literature such as~\cite{Ismail-Muldoon-119, Muldoon-78, haer}, \cite[Chapter~XIII]{mpf-93} and the related references therein are actually logarithmically completely monotonic.
\end{rem}

\section{Lemmas}

In order to verify Theorem~\ref{itsf-gao-gen-3}, Theorem~\ref{psi-plus-comp-mon-3} and their corollaries in Section~\ref{sec-1.4} and Section~\ref{sec-1.5}, we need the following lemmas, in which Lemma~\ref{mentioned} is simple but has been validated in \cite{notes-best-simple-equiv.tex, notes-best-simple-open.tex, notes-best-simple.tex, sandor-gamma-2-ITSF.tex} to be especially effectual in proving the monotonicity and (logarithmically) complete monotonicity of functions involving the gamma, psi and polygamma functions.

\begin{lem}\label{mentioned}
Let $f(x)$ is a function defined on an infinite interval $I$ whose right end is $\infty$. If $\lim_{x\to\infty}f(x)=\delta$ and $f(x)-f(x+\varepsilon)>0$ hold true for  some given scalar $\varepsilon>0$ and all $x\in I$, then $f(x)>\delta$.
\end{lem}

\begin{proof}
By mathematical induction, for all $x\in I$, we have
\begin{equation*}
f(x)>f(x+\varepsilon)>f(x+2\varepsilon)>\dotsm>f(x+k\varepsilon)\to\delta
\end{equation*}
as $k\to\infty$. The proof of Lemma~\ref{mentioned} is complete.
\end{proof}

\begin{lem}[\cite{abram,wange,wang}]
The polygamma functions $\psi^{(k)}(x)$ may be expressed for $x>0$ and
$k\in\mathbb{N}$ as
\begin{equation}\label{psim}
\psi ^{(k)}(x)=(-1)^{k+1}\int_{0}^{\infty}\frac{t^{k}e^{-xt}}{1-e^{-t}}\td t.
\end{equation}
For $x>0$ and $r>0$,
\begin{equation}\label{fracint}
\frac1{x^r}=\frac1{\Gamma(r)}\int_0^\infty t^{r-1}e^{-xt}\td t.
\end{equation}
For $i\in\mathbb{N}$ and $x>0$,
\begin{equation}\label{psisymp4}
\psi^{(i-1)}(x+1)=\psi^{(i-1)}(x)+\frac{(-1)^{i-1}(i-1)!}{x^i}.
\end{equation}
\end{lem}

\begin{lem}\label{comp-thm-1}
For $k\in\mathbb{N}$, the double inequality
\begin{equation}\label{qi-psi-ineq}
\frac{(k-1)!}{x^k}+\frac{k!}{2x^{k+1}}<(-1)^{k+1}\psi^{(k)}(x)
<\frac{(k-1)!}{x^k}+\frac{k!}{x^{k+1}}
\end{equation}
holds on $(0,\infty)$.
\end{lem}

\begin{proof}
In~\cite[Theorem~2.1]{Ismail-Muldoon-119} and \cite[Lemma~1.3]{sandor-gamma-2-ITSF.tex}, the function $\psi(x)-\ln x+\frac{\alpha}x$ was proved to be completely monotonic on $(0,\infty)$, i.e.,
\begin{equation}\label{com-psi-ineq-dfn}
(-1)^i\biggl[\psi(x)-\ln x+\frac{\alpha}x\biggr]^{(i)}\ge0
\end{equation}
for $i\ge0$, if and only if $\alpha\ge1$, so is its negative, i.e., the inequality~\eqref{com-psi-ineq-dfn} is reversed, if and only if $\alpha\le\frac12$. In \cite{chen-qi-log-jmaa} and \cite[Theorem~2.1]{Muldoon-78}, the function $\frac{e^x\Gamma(x)} {x^{x-\alpha}}$ was proved to be logarithmically completely monotonic on $(0,\infty)$, i.e.,
\begin{equation}\label{com-psi-ineq-ln-dfn}
(-1)^k\biggl[\ln\frac{e^x\Gamma(x)} {x^{x-\alpha}}\biggr]^{(k)}\ge0
\end{equation}
for $k\in\mathbb{N}$, if and only if $\alpha\ge1$, so is its reciprocal, i.e., the inequality~\eqref{com-psi-ineq-ln-dfn} is reversed, if and only if $\alpha\le\frac12$. Considering the fact \cite[p.~82]{e-gam-rat-comp-mon} that a completely monotonic function which is non-identically zero cannot vanish at any point on $(0,\infty)$ and rearranging either~\eqref{com-psi-ineq-dfn} for $i\in\mathbb{N}$ or~\eqref{com-psi-ineq-ln-dfn} for $k\ge2$ leads to the double inequality~\eqref{qi-psi-ineq} immediately.
\end{proof}

\section{Proofs of theorems and corollaries}

\begin{proof}[Proof of Theorem~\ref{itsf-gao-gen-3}]
It is a standard argument to obtain that the function $\delta(t)$ is strictly decreasing from $(0,\infty)$ onto $\bigl(0,\frac12\bigr)$.
\par
Let $g_{i,\alpha,\beta}(x)=x^\alpha\bigl\lvert\psi^{(i)}(x+\beta)\bigl\lvert$ on $(0,\infty)$. Direct calculation and rearrangement yields
\begin{equation}\label{g-x-beta-1}
\begin{split}
\frac{g_{i,\alpha,\beta}'(x)}{x^{\alpha-1}}
&=\alpha\bigl\lvert\psi^{(i)}(x+\beta)\bigl\lvert-x\bigl\lvert\psi^{(i+1)}(x+\beta)\bigl\lvert\\
&=(-1)^{i+1}\bigl[\alpha\psi^{(i)}(x+\beta)+x\psi^{(i+1)}(x+\beta)\bigr].
\end{split}
\end{equation}
Making use of~\eqref{qi-psi-ineq} in~\eqref{g-x-beta-1} gives
\begin{equation}\label{lim=0-1}
\lim_{x\to\infty}\frac{g_{i,\alpha,\beta}'(x)}{x^{\alpha-1}}=0
\end{equation}
for $i\in\mathbb{N}$, $\alpha\in\mathbb{R}$ and $\beta\ge0$. In virtue of formulas~\eqref{psisymp4}, \eqref{fracint} and~\eqref{psim} in sequence, straightforward computation reveals
\begin{gather}
\frac{g_{i,\alpha,\beta}'(x)}{x^{\alpha-1}}-\frac{g_{i,\alpha,\beta}'(x+1)}{(x+1)^{\alpha-1}} =(-1)^{i+1}\bigl\{\alpha\bigl[\psi^{(i)}(x+\beta)-\psi^{(i)}(x+\beta+1)\bigr] \notag\\
+x\bigl[\psi^{(i+1)}(x+\beta)-\psi^{(i+1)}(x+\beta+1)\bigr]-\psi^{(i+1)}(x+\beta+1)\bigr\} \notag\\
\begin{aligned}
&=\frac{i!\alpha}{(x+\beta)^{i+1}} -\frac{(i+1)!x}{(x+\beta)^{i+2}}
-\frac{(i+1)!}{(x+\beta)^{i+2}} +(-1)^{i+2}\psi^{(i+1)}(x+\beta)\\
&=(-1)^{i+2}\psi^{(i+1)}(x+\beta)+\frac{i!(\alpha-i-1)}{(x+\beta)^{i+1}}
+\frac{(i+1)!(\beta-1)}{(x+\beta)^{i+2}}\\
&=\int_0^\infty\biggl[\frac{t}{1-e^{-t}}+(\beta-1)t+\alpha-i-1\biggr]
t^ie^{-(x+\beta)t}\td t\\
&\triangleq\int_0^\infty h_{i,\alpha,\beta}(t) t^ie^{-(x+\beta)t}\td t.\label{3-key-eq}
\end{aligned}
\end{gather}
\par
For $\beta=0$, easy differentiation shows that $h_{i,\alpha,0}'(t)=-\delta(t)<0$, and so the function $h_{i,\alpha,0}(t)$ is strictly decreasing from $(0,\infty)$ onto $(\alpha-i-1,\alpha-i)$. Thus, if $\alpha\ge i+1$, the functions $h_{i,\alpha,0}(t)$ and
$$
\frac{g_{i,\alpha,0}'(x)}{x^{\alpha-1}} -\frac{g_{i,\alpha,0}'(x+1)}{(x+1)^{\alpha-1}}
$$
are positive on $(0,\infty)$. Combining this with~\eqref{lim=0-1} and considering Lemma~\ref{mentioned}, it is deduced that the functions $\frac{g_{i,\alpha,0}'(x)}{x^{\alpha-1}}$ and $g_{i,\alpha,0}'(x)$ are positive on $(0,\infty)$. Hence, the function $g_{i,\alpha,0}(x)$ is strictly increasing on $(0,\infty)$ for $\alpha\ge i+1$. Similarly, for $\alpha\le i$, the function $g_{i,\alpha,0}(x)$ is strictly decreasing on $(0,\infty)$.
\par
For $\beta>0$, it is easy to see that $h_{i,\alpha,\beta}'(t)=-\delta(t)+\beta$, and $h_{i,\alpha,\beta}'(t)$ is strictly increasing from $(0,\infty)$ onto $\bigl(\beta-\frac12,\beta\bigr)$.
Consequently, if $\beta\ge\frac12$, the function $h_{i,\alpha,\beta}'(t)$ is positive and $h_{i,\alpha,\beta}(t)$ is strictly increasing from $(0,\infty)$ onto $(\alpha-i,\infty)$. Accordingly, if $\alpha\ge i$ and $\beta\ge\frac12$, the function $h_{i,\alpha,\beta}(t)$ is positive on $(0,\infty)$, that is,
\begin{equation}\label{minus-g-positive}
\frac{g_{i,\alpha,\beta}'(x)}{x^{\alpha-1}}-\frac{g_{i,\alpha,\beta}'(x+1)}{(x+1)^{\alpha-1}}>0
\end{equation}
on $(0,\infty)$. Combining this with Lemma~\ref{mentioned} results in the positivity of $g_{i,\alpha,\beta}'(x)$ on $(0,\infty)$. Therefore, for $\alpha\ge i$ and $\beta\ge\frac12$, the function $g_{i,\alpha,\beta}(x)$ is strictly increasing on $(0,\infty)$.
\par
For $0<\beta<\frac12$, since $h_{i,\alpha,\beta}'(t)$ is strictly increasing from $(0,\infty)$ onto $\bigl(\beta-\frac12,\beta\bigr)$, the function $h_{i,\alpha,\beta}(t)$ attains its unique minimum at some point $t_0\in(0,\infty)$ with $\delta(t_0)=\beta$. As a result, the unique minimum of $h_{i,\alpha,\beta}(t)$ equals
$$
\frac{\delta^{-1}(\beta)e^{\delta^{-1}(\beta)}}{e^{\delta^{-1}(\beta)}-1}
+(\beta-1)\delta^{-1}(\beta)+\alpha-i-1,
$$
where $\delta^{-1}$ is the inverse function of $\delta$ and is strictly decreasing from $(0,\frac12)$ onto $(0,\infty)$. Consequently, when the inequality~\eqref{alpha-beta-inverse} holds for $0<\beta<\frac12$, the function $h_{i,\alpha,\beta}(t)$ is positive on $(0,\infty)$, which means that the inequality~\eqref{minus-g-positive} holds true. Accordingly, making use of the limit~\eqref{lim=0-1} and Lemma~\ref{mentioned} again yields that the function $g_{i,\alpha,\beta}(x)$ is strictly increasing on $(0,\infty)$ if $0<\beta<\frac12$ and the inequality~\eqref{alpha-beta-inverse} is valid. The sufficiency is proved.
\par
If $g_{i,\alpha,0}(x)$ is strictly decreasing on $(0,\infty)$, then
\begin{equation}\label{deralp1}
x^{i+1-\alpha}g_{i,\alpha,0}'(x)
=\alpha{x^{i}}\bigl\lvert\psi^{(i)}(x)\bigl\lvert-x^{i+1}\bigl\lvert\psi^{(i+1)}(x)\bigl\lvert<0.
\end{equation}
Applying~\eqref{qi-psi-ineq} in~\eqref{deralp1} and letting $x\to\infty$ lead to
\begin{align*}
0&\ge\lim_{x\to\infty}x^{i+1-\alpha}g'_{i,\alpha,0}(x)\\
&\ge\alpha\lim_{x\to\infty}x^i\biggl[\frac{(i-1)!}{x^i}+\frac{i!}{2x^{i+1}}\biggr]
-\lim_{x\to\infty}x^{i+1}\biggl[\frac{i!}{x^{i+1}}+\frac{(i+1)!}{x^{i+2}}\biggr]\\
&=(i-1)!(\alpha-i),
\end{align*}
which means $\alpha\le i$.
\par
If $g_{i,\alpha,0}(x)$ is strictly increasing on $(0,\infty)$, then
\begin{equation}\label{deralp2}
x^{i+2-\alpha}g_{i,\alpha,0}'(x)
=\alpha{x^{i+1}}\bigl\lvert\psi^{(i)}(x)\bigl\lvert-x^{i+2}\bigl\lvert\psi^{(i+1)}(x)\bigl\lvert>0.
\end{equation}
Employing~\eqref{psisymp4} and~\eqref{qi-psi-ineq} in~\eqref{deralp2} and taking $x\to\infty$ results in
\begin{align*}
0&\le\lim_{x\to0^+}x^{i+2-\alpha}g_{i,\alpha,0}'(x)\\
&=\lim_{x\to0^+}\biggl\{\alpha{x^{i+1}}\bigl\lvert\psi^{(i)}(x)\bigl\lvert
-x^{i+2}\biggl[\bigl\lvert\psi^{(i+1)}(x+1)\bigl\lvert+\frac{(i+1)!}{x^{i+2}}\biggr]\biggr\}\\
&=\alpha\lim_{x\to0^+}{x^{i+1}}\bigl\lvert\psi^{(i)}(x)\bigl\lvert-(i+1)!
-\lim_{x\to0^+}x^{i+2}\bigl\lvert\psi^{(i+1)}(x+1)\bigl\lvert\\
&\le\alpha\lim_{x\to0^+}{x^{i+1}}
\biggl[\frac{(i-1)!}{x^i}+\frac{i!}{x^{i+1}}\biggr]-(i+1)!\\
&\quad-\lim_{x\to0^+}x^{i+2}\biggl[\frac{i!}{(x+1)^{i+1}}
+\frac{(i+1)!}{2(x+1)^{i+2}}\biggr]\\
&=i!(\alpha-i-1),
\end{align*}
thus, the necessary condition $\alpha\ge i+1$ follows.
\par
If the function $g_{i,\alpha,\beta}(x)$ is strictly increasing on $(0,\infty)$
for $\beta>0$, then
\begin{equation}\label{deralp11}
x^{i+1-\alpha}g_{i,\alpha,\beta}'(x)
=\alpha{x^{i}}\bigl\lvert\psi^{(i)}(x+\beta)\bigl\lvert-x^{i+1}\bigl\lvert\psi^{(i+1)}(x+\beta)\bigl\lvert>0.
\end{equation}
Utilizing~\eqref{qi-psi-ineq} in~\eqref{deralp11} and taking limit gives
\begin{align*}
0&\le\lim_{x\to\infty}x^{i+1-\alpha}g'_{i,\alpha,\beta}(x)\\*
&\le\alpha\lim_{x\to\infty}x^i\biggl[\frac{(i-1)!}{(x+\beta)^{i}}
+\frac{i!}{(x+\beta)^{i+1}}\biggr] -\lim_{x\to\infty}x^{i+1}\biggl[\frac{i!}{(x+\beta)^{i+1}}
+\frac{(i+1)!}{2(x+\beta)^{i+2}}\biggr]\\
&=(i-1)!(\alpha-i),
\end{align*}
which is equivalent to $\alpha\ge i$. The proof of Theorem~\ref{itsf-gao-gen-3} is thus completed.
\end{proof}

\begin{rem}
The first two conclusions in Theorem~\ref{itsf-gao-gen-3} were ever proved by virtue of the convolution theorem for Laplace transforms in~\cite[Lemma~1]{alzeraeq} and~\cite[Lemma~2.2]{forum-alzer}, so we supply a new and unified proof for them here.
\end{rem}

\begin{proof}[Proof of Corollary~\ref{bernou-thm}]
It is well-known~\cite{abram,wange,wang} that Bernoulli polynomials $B_k(x)$ may be defined by
\begin{equation}
\frac{te^{xt}}{e^t-1}=\sum_{k=0}^\infty B_k(x)\frac{t^k}{k!}
\end{equation}
and that Bernoulli numbers $B_k$ and Bernoulli polynomials $B_k(x)$ are connected by $B_k(1)=(-1)^kB_k(0)=(-1)^kB_k$ and $B_{2k+1}(0)=B_{2k+1}=0$ for $k\ge1$. Using these notations, the functions $h_{i,\alpha,\beta}(t)$ and $h_{i,\alpha,\beta}'(t)$ may be rewritten as
\begin{align*}
h_{i,\alpha,\beta}(t)&=\frac{te^t}{e^t-1}+(\beta-1)t+\alpha-i-1\\
&=\alpha-i+\biggl(\beta-\frac12\biggr)t+\sum_{k=2}^\infty B_k(1)\frac{t^k}{k!}\\
&=\alpha-i+\biggl(\beta-\frac12\biggr)t+\sum_{k=2}^\infty (-1)^kB_k\frac{t^k}{k!}\\
&=\alpha-i+\biggl(\beta-\frac12\biggr)t+\sum_{k=1}^\infty (-1)^{k+1}B_{k+1}\frac{t^{k+1}}{(k+1)!}\\
&=\alpha-i+\biggl(\beta-\frac12\biggr)t+\sum_{k=0}^\infty B_{2k+2}\frac{t^{2k+2}}{(2k+2)!}
\end{align*}
and
\begin{equation*}
h_{i,\alpha,\beta}'(t)=\beta-\frac12+\sum_{k=1}^\infty B_{2k}\frac{t^{2k-1}}{(2k-1)!}.
\end{equation*}
The proof of Theorem~\ref{itsf-gao-gen-3} shows that
\begin{enumerate}
\item
$h_{i,\alpha,0}'(t)<0$ on $(0,\infty)$;
\item
$h_{i,\alpha,0}(t)>0$ on $(0,\infty)$ if $\alpha\ge i+1$;
\item
$h_{i,\alpha,0}(t)<0$ on $(0,\infty)$ if $0<\alpha\le i$;
\item
$h_{i,\alpha,\beta}'(t)>0$ on $(0,\infty)$ if $\beta\ge\frac12$;
\item
$h_{i,\alpha,\beta}(t)>0$ on $(0,\infty)$ if $\alpha\ge i$ and $\beta\ge\frac12$;
\item
$h_{i,\alpha,\beta}(t)>0$ on $(0,\infty)$ if $0<\beta<\frac12$ and inequality~\eqref{alpha-beta-inverse} holds true.
\end{enumerate}
Basing on these and by standard argument, Corollary~\ref{bernou-thm} is thus proved.
\end{proof}

\begin{proof}[Proof of Theorem~\ref{psi-plus-comp-mon-3}]
If $h_{i,\alpha,\beta}(t)\gtrless0$ on $(0,\infty)$, then the function
$$
\pm\int_0^\infty h_{i,\alpha,\beta}(t)t^{i}e^{-(x+\beta)t}\td t
$$
is completely monotonic on $(-\beta,\infty)$, and so, by virtute of~\eqref{3-key-eq}, it is derived that
$$
\pm\biggl[\frac{g_{i,\alpha,\beta}'(x)}{x^{\alpha-1}}
-\frac{g_{i,\alpha,\beta}'(x+1)}{(x+1)^{\alpha-1}}\biggr]
$$
is completely monotonic on $(0,\infty)$, that is,
\begin{multline}\label{j=0=1}
(-1)^j\biggl[\frac{g_{i,\alpha,\beta}'(x)}{x^{\alpha-1}}
-\frac{g_{i,\alpha,\beta}'(x+1)}{(x+1)^{\alpha-1}}\biggr]^{(j)}\\*
=(-1)^j\biggl[\frac{g_{i,\alpha,\beta}'(x)}{x^{\alpha-1}}\biggr]^{(j)}
-(-1)^j\biggl[\frac{g_{i,\alpha,\beta}'(x+1)}{(x+1)^{\alpha-1}}\biggr]^{(j)}\gtreqless0
\end{multline}
on $(0,\infty)$ for $j\ge0$. Moreover, formulas~\eqref{qi-psi-ineq} and
\eqref{g-x-beta-1} imply
\begin{equation}\label{j=0}
\lim_{x\to\infty}\biggl[\frac{g_{i,\alpha,\beta}'(x)}{x^{\alpha-1}}\biggr]^{(j)}
=\lim_{x\to\infty}(-1)^j\biggl[\frac{g_{i,\alpha,\beta}'(x)}{x^{\alpha-1}}\biggr]^{(j)}=0.
\end{equation}
Combining~\eqref{j=0=1} and~\eqref{j=0} with Lemma~\ref{mentioned} concludes that
$$
(-1)^j\biggl[\frac{g_{i,\alpha,\beta}'(x)}{x^{\alpha-1}}\biggr]^{(j)}\gtreqless0,
$$
that is, the function
$$
\pm\frac{g_{i,\alpha,\beta}'(x)}{x^{\alpha-1}}
=\pm\bigl[\alpha\bigl\lvert\psi^{(i)}(x+\beta)\bigl\lvert -x\bigl\lvert\psi^{(i+1)}(x+\beta)\bigl\lvert\bigr]
$$
is completely monotonic on $(0,\infty)$, if $h_{i,\alpha,\beta}(t)\gtrless0$ on $(0,\infty)$. In the proof of Theorem~\ref{itsf-gao-gen-3}, we have demonstrated that $h_{i,\alpha,\beta}(t)$ is positive on $(0,\infty)$ if either $\beta=0$ and $\alpha\ge i+1$, or $\beta\ge\frac12$ and $\alpha\ge i$, or $0<\beta<\frac12$ and the inequality~\eqref{alpha-beta-inverse} is satisfied, and that
$h_{i,\alpha,\beta}(t)$ is negative on $(0,\infty)$ if $\beta=0$ and
$\alpha\le i$. As a result, the sufficient conditions for the function
$
\alpha\bigl\lvert\psi^{(i)}(x+\beta)\bigl\lvert-x\bigl\lvert\psi^{(i+1)}(x+\beta)\bigl\lvert
$
to be completely monotonic on $(0,\infty)$ follow.
\par
The derivation of necessary conditions is same as in Theorem~\ref{itsf-gao-gen-3}. The proof of Theorem~\ref{psi-plus-comp-mon-3} is complete.
\end{proof}

\begin{proof}[Proof of Corollary~\ref{cor-sub-3}]
It follows easily from Theorem~\ref{psi-plus-comp-mon-3} and the facts that
$$
\pm\biggl[\frac{\alpha}{x}\bigl\lvert\psi^{(i)}(x+\beta)\bigl\lvert
-\bigl\lvert\psi^{(i+1)}(x+\beta)\bigl\lvert\biggr]
=\pm\frac1x\bigl\{\alpha\bigl\lvert\psi^{(i)}(x+\beta)\bigl\lvert
-x\bigl\lvert\psi^{(i+1)}(x+\beta)\bigl\lvert\bigr\},
$$
that the function $\frac1x$ is completely monotonic on $(0,\infty)$, and that the product of any finite completely monotonic functions is also completely monotonic on the intersection of their domains.
\end{proof}

\end{document}